\documentclass[]{amsart}
\usepackage[lmargin=3cm,tmargin=3cm,bmargin=3cm,rmargin=3cm]{geometry} 
\usepackage[table]{xcolor}
\usepackage{tabularx}
\usepackage{mathrsfs,amsmath, amsfonts, amssymb, amscd}
\usepackage{amsaddr}
\usepackage{graphicx,graphics,epsfig}
\usepackage{subfig}
\usepackage{amssymb}
\usepackage{amsfonts}
\usepackage{amsmath}
\usepackage{epsfig}
\usepackage{float}
\usepackage{setspace}
\usepackage{tikz}
\usepackage{enumerate}
\usepackage[colorlinks=true,urlcolor=blue, citecolor=red,linkcolor=blue,linktocpage,pdfpagelabels, bookmarksnumbered,bookmarksopen]{hyperref}
\usepackage{marginnote}
\usepackage{epstopdf}

\newtheorem{theorem}{Theorem}[section]

\newtheorem{corollary}[theorem]{Corollary}
\newtheorem{lemma}[theorem]{Lemma}
\newtheorem{remark}[theorem]{Remark}

\numberwithin{equation}{section}

\newcommand{\p}{\partial}

\newcommand{\cC}{{\mathcal C}}

\newcommand{\cT}{{\mathcal T}}
\newcommand{\cM}{{\mathcal M}}
\newcommand{\cK}{{\mathcal K}}
\newcommand{\cE}{{\mathcal E}}
\newcommand{\R}{{\mathbb R}}

\title[Adiabatic limit of Benney type systems]{Uniform adiabatic limit of Benney type systems}

\author[A. J. Corcho]{\bf Ad\'an J. Corcho}
\address{Instituto de Matem\'atica, Universidade Federal do Rio de Janeiro.\\ 
Centro de Tecnologia - Bloco C. Cidade Universit\'aria.\\
Ilha do Fund\~ao  21941-909.  Rio de Janeiro - RJ, Brazil.}
\email{adan@im.ufrj.br}

\author[J. C. Cordero]{\bf Juan C.  Cordero}
\address{Departamento de Matem\'aticas y Estad\'istica\\ 
Universidad Nacional de Colombia\\
Campus La Nubia - Manizales}
\email{jccorderoc@unal.edu.co}

\thanks{A. J. Corcho was partially supported by CAPES and CNPq (307761/2016-9), Brazil}
\thanks{J. C. Cordero was partially supported by the Departamento de Matem\'aticas y Estad\'istica, Universidad Nacional de Colombia, Sede Manizales}


\date{\today}

\begin{document}

\maketitle

\setcounter{page}{1}

\begin{quote}
\textbf{Abstract.}
{\small
In this paper we show that solutions of the cubic nonlinear Schr\"odinger equation  are asymptotic limit of solutions  to the Benney system.  Due to the special characteristic of the one-dimensional transport equation same result is obtained for solutions of the one-dimensional Zakharov and 1d-Zakharov-Rubenchik systems. Convergence is reached in the topology
$L^2(\R)\times L^2(\R)$ and with an approximation in the energy space $H^1(\R)\times L^2(\R)$. In the case of the Zakharov system this is achieved without the condition $\partial_t n(x,0) \in \dot H^{-1}(\R)$ for the wave component, improving previous results. 
}
\end{quote}

\section{Introduction}
We consider a family of one-dimensional nonlinear dispersive systems, given by the following coupling equations: 
\begin{equation}\label{S-GBenney}
\begin{cases}
i\partial_tu+ \partial_x^2u= (\tau \vert u\vert^2+\alpha v+  \alpha' z)u, & (x,t)\in \mathbb{R}\times \mathbb{R}^+,\\
\varepsilon \partial_tv + \lambda \partial_xv= \beta\partial_x|u|^2,&\\
\varepsilon \partial_tz + \lambda'\partial_xz= \beta' \partial_x|u|^2,
\end{cases}
\end{equation}
where $u$ is a complex-valued function, $v$ and $z$ are real-valued functions, the physical para\-meters $\tau, \alpha, \alpha'$, $\lambda, \lambda', \beta, \beta'$  are real numbers, and $0 < \varepsilon<1$. This model governs, on certain parameter regimes, the dynamics of many physical phenomena and it is in the ``neighborhood" of some other important models of the mathematical-physics; for example, the Zakharov system, the Davey-Stewartson system and the nonlinear Schr\"odinger equation. We give further information about well-posedness concerning System \eqref{S-GBenney} in Section \ref{Final-Remarks}. 

\medskip 
The family  \eqref{S-GBenney} contains, for instance,  some cases of the non resonant dynamics of small amplitude Alfven waves propagating in a plasma \cite{C-L-P-S, Oliveira1}, modeled by the coupled equations: 
\begin{equation}\label{ZR}
\begin{cases}
i\partial_tu+\partial_x^2u= k(c\vert u\vert^2-\tfrac{1}{2}a \rho+ \varphi)u,\\
\varepsilon \partial_t\rho + \partial_x(\varphi-a\rho)= -k\partial_x|u|^2,&\\
\varepsilon \partial_t\varphi +\partial_x(b\rho-a\varphi)= \tfrac{1}{2}k\partial_x|u|^2,
\end{cases}
\end{equation}
where we have taken the frequency $\omega$ equal to $1$, on the expanded flat wave front to generate $u$. This model is known as 1d-Zakharov-Rubenchik type system,
which in the case  $b>0$ and $b-a^2\ne0$, using the transformation (see \cite{Linares-Matheus})
\begin{equation}
\rho=\psi_1+\psi_2,\quad \varphi=\sqrt{b}(\psi_1-\psi_2),
\end{equation}
can be rewritten as 
\begin{equation}\label{1dZR}
\begin{cases}
i\partial_tu+\partial_x^2u= (c\vert u\vert^2-(\sqrt{b}+\tfrac{a}{2})\psi_2+ (\sqrt{b}-\tfrac{a}{2})\psi_1)u,\\
\varepsilon \partial_t\psi_1 +(\sqrt{b}-a)\partial_x\psi_1= \frac{1}{2}(-1+\tfrac{a}{2\sqrt{b}}) \partial_x|u|^2,&\\
\varepsilon \partial_t\psi_2 -(a+\sqrt{b})\partial_x\psi_2= \frac{1}{2}(-1-\tfrac{a}{2\sqrt{b}})\partial_x|u|^2.
\end{cases}
\end{equation}

\medskip 
Another system included in the family \eqref{S-GBenney} is the 1d-Zakharov system describing Langmuir turbulence \cite{Zakharov}, given by
\begin{equation}\label{Z}
\begin{cases}
i\partial_tu+\partial_x^2u=nu,\medskip \\
\varepsilon^2\partial_t^2n -\partial_x^2n=\partial_x^2\vert u\vert^2,
\end{cases}
\end{equation} 
where $0<\varepsilon=k/c_s<1$, $k$ is a positive parameter and $c_s$ the ionic sound speed. It can be set in the form of \eqref{S-GBenney} because we can write the wave equation of this system as
\begin{equation}
(\varepsilon\partial_{t}-\partial_{x})(\varepsilon\partial_{t}+\partial_{x})n=\partial_x^2\vert u\vert
\end{equation}
and we make
 \begin{equation}
(\varepsilon\partial_t\pm\partial_x)n=\partial_xn_{\mp}
 \end{equation} 
to  consider the two traveling wave profiles. Then we have
\begin{equation}\label{Zred}
\begin{cases}
i\partial_tu+ \partial_x^2u= \frac{1}{2} (n_{-}-n_{+})u,\\
\varepsilon\partial_tn_+ + \partial_xn_+= \partial_x|u|^2,&\\
\varepsilon\partial_tn_- -\partial_xn_-= \partial_x|u|^2,
\end{cases}
\end{equation}
with 
\begin{equation}
\frac{1}{2}(n_{-}-n_{+})=n.
\end{equation}

In \cite{G-T-V} a similar change of variables was performed for the Zakharov system, with complex positive and negative frequency parts for the wave equation, that is
\begin{equation}
n_{\pm}=n\pm i\omega^{-1}\partial_tn,\quad \omega=(-\Delta)^{1/2}.
\end{equation}  

This allowed them to implement the Bourgain method to prove well posedness for the associated Cauchy problem in a wide class of Sobolev regularity for initial data.

\medskip
From the physical point of view, the models \eqref{ZR} and \eqref{Z} are magneto-hydrodynamics type systems in plasma physics \cite{Rubenchik-Zakharov, Kuznetsov-Zakharov}, however \eqref{ZR} is known as the Benney-Roskes system in the context of gravitational water waves \cite{Benney-Roskes}. 

\medskip 
A simpler model in the study of a general theory of water waves interactions  in a nonlinear medium \cite{Benney1,Benney2}, is the Benney system
\begin{equation}\label{S-Benney}
\begin{cases}
i\partial_tu+\partial_x^2u=\alpha uv,\medskip \\
\varepsilon \partial_tv + \lambda \partial_xv=\beta\partial_x|u|^2,&
\end{cases}
\end{equation}
where $\vert\lambda\vert = 1$.

\medskip
Before summarize some useful preliminary results we fix some notations. For a fixed positive time $T$, we are going to write
\begin{equation}
\Vert f\Vert_{L^\infty_T}:=\sup_{t\in[0,T]}\vert f(t)\vert, \quad \Vert f\Vert_{L^\infty_t}:=\sup_{t\ge0} \vert f(t)\vert
\end{equation}
and the symbol $\Vert \cdot \Vert_{L_t^pL_x^q}$ or $\Vert \cdot \Vert_{L_T^pL_x^q}$ will indicate the usual norm of a mixed space for $t\in \R^+$ or $t\in [0, T]$, respectively.

\medskip
We are interested in initial data belonging  to $H^s$, which denotes the classical $L^2$-Sobolev space and  also will be used the mixed norm of $L^\infty([0,T];H^s(\mathbb{R}))$, defined  by 
\begin{equation}
\Vert f\Vert_{L^{\infty}_TH^s_x}:=\sup_{0\leq t\leq T}\Vert f(\cdot,t)\Vert_{H^s}.
\end{equation}	

\medskip 	
Well-posedness for the Cauchy problem associated to the above models in the $H^s(\R^d)$ spaces, as well as other important properties of the dynamics of the solutions, have been extensively considered for many authors. 

\medskip
Regarding the system \eqref{ZR} there are a few works but this one started to get attention in the last years, see for instance \cite{Oliveira1, Oliveira2, Linares-Matheus}. As far as we know the best local and global well-posedness for \eqref{ZR} was established in \cite{Linares-Matheus} and in dimensions $d=2,3$ we refer to the works \cite{Ponce-Saut, Cordero, Luong-Mauser-Saut, Cordero-Quintero} for recent advances concerning existence of solutions, asymptotic behavior, instability of standing waves and blow-up solutions. 

\medskip
	About the existence of global solutions for \eqref{Z} there are global weak solutions for initial data $(u_0, n_0, n_1)\in H^1\times L^2\times H^{-1}$ (with $n_1(x)=\partial_tn(x,0)$)  and smooth solutions 
	\begin{equation}
	u\in L^\infty([0,T];\, H^m),\quad n\in L^\infty([0,T];\, H^{m-1}),
	\end{equation}
	for any time $T>0$ and for initial data in $H^m\times H^{m-1}\times H^{m-2}$ with $m\ge3$ (see \cite{Sulem-Sulem}). Also, when \eqref{Z} is written as in \eqref{HamZ}, the solutions $(u_\varepsilon, n_\varepsilon, v_\varepsilon)\in H^3\times H^2\times H^1$ associated to the bounded family of initial data are bounded uniformly respect to $\varepsilon$ in $H^1\times L^2\times L^2$. This fact was used in  \cite{Added-Added} to prove the weak convergence to the cubic nonlinear Schr\"odinger (cubic-NLS) equation.
A local and global theory in all dimensions for the system \eqref{Z} was established in \cite{G-T-V}, and improvements of this in the two dimensional case were established later in \cite{B-H-H-T} and in one dimensional case below energy space in \cite{Pecher1, Pecher}. 

\medskip
On the other hand, in \cite{Schochet-Weinstein} it was proven the convergence of solutions of the 2$d$ and 3$d$ Zakharov system to the corresponding solutions of the cubic-NLS equation in the subsonic limit ($\varepsilon\to0$), more exactly,  they proved:

\begin{theorem}[\cite{Schochet-Weinstein}]
Let $m\ge[d/2]+3$ $(d=2,3)$, $\partial_tn_\varepsilon(x,0)=\nabla w_{0\varepsilon}$ and  $T$ the time of existence of solutions (independent of $\varepsilon$). 
Assume that  
\begin{align*}
&\Vert u_{0\varepsilon}\Vert_{H^{m+1}}+ \sqrt{\varepsilon}\Vert w_{0\varepsilon}\Vert_{H^m} +\Vert n_{0\varepsilon}\Vert_{H^m}\leq C,\\ 
&\frac{1}{\sqrt{\varepsilon}}\Vert \nabla(n_{0\varepsilon}+\vert u_{0\varepsilon}\vert^2)\Vert_{H^{m-1}}+\Vert \nabla w_{0\varepsilon}\Vert_{H^{m-1}}\leq C,
\end{align*}
and 
$$\lim\limits_{\varepsilon \to 0}\Vert u_{0\varepsilon}-u_0\Vert_{H^{m+1}}=0.$$ Then, we have 
\begin{align}
&n_\varepsilon+\vert u_\varepsilon\vert^2\to0\quad\text{in}\quad \cC([0,T]\times\mathbb{R}^d),\medskip \\
&\nabla\big[ n_\varepsilon+\vert u_\varepsilon\vert^2\big]\to0\quad\text{in}\quad \cC([0,T];H^{m-2}),\medskip\\
& u_\varepsilon-u\to0\quad\text{in} \quad \cC^1([0,T]\times\mathbb{R}^d)\cap \cC^1([0,T]; \cC^2).
\end{align}
\end{theorem}

\medskip 
For solutions with small amplitude were obtained rates of this convergence in \cite{Added-Added} for $d=1,2,3$. For example, when the dimension is $d=1$, we have:
\begin{theorem}[\cite{Added-Added}]
Consider de Cauchy problem associated to the system \eqref{Z}. Let $m\ge 3$ and assume that initial data satisfies
$$u_0\in H^{m+2},\; n_0\in H^{m+1}\; \text{and}\;  n_1=\partial_tn(x,0)\in H^m\; (n_1=\partial_x^2w_0\; \text{with}\;  \partial_x w_0\in L^2).$$
Then, 
\begin{align*}
&\Vert u_\varepsilon(\cdot,t)-u(\cdot,t)\Vert_{H^m}\leq M(t)(\varepsilon+\sqrt{\varepsilon}\Vert n_0+\vert u_0\vert^2 \Vert_{H^m}),\\  \intertext{and}
&\Vert n_\varepsilon(\cdot,t)+\vert u(\cdot,t)\vert^2-\tilde{n}(\cdot,t/\varepsilon)\Vert_{H^{m-1}}\leq M(t)(\varepsilon+\sqrt{\varepsilon}\Vert n_0+\vert u_0\vert^2 \Vert_{H^m}),
\end{align*}
for some function $M(t)\in L^{\infty}_{loc}(\mathbb{R}^+)$,  where $$i\partial_tu+\partial_x^2u=-\vert u\vert^2u,\quad u(x,0)=u_0(x),$$ 
and $\tilde{n}$ is a fitting corrector satisfying the wave equation
$$
\begin{cases}
\partial_{tt}\tilde{n}-\partial_x^2\tilde{n}=0,\\
\tilde{n}(x,0)=n_0(x)+\vert u_0(x)\vert^2, \quad \partial_t\tilde{n}(x,0)=0.
\end{cases}
$$ 
\end{theorem}

In \cite{Ozawa-Tsutsumi}, it was found optimal rates for this convergence. Also, in \cite{KPV} the convergence in this limit has been proved whenever the initial data are uniformly bound in $H^5$. After, in \cite{Masmoudi-Nakanishi} the convergence was proven in the energy space, still maintaining the condition $n_1\in\dot H^{-1}$ for the solutions found in \cite{Bourgain-Colliander} in dimension $d=3$, and $\vert \varepsilon\nabla^{-1}\vert n_1$ decaying for high frequency. 

\medskip
Finally, for system \eqref{S-Benney} in the works \cite{B-O-P-1, B-O-P-2, Corcho, G-T-V, Laurencot,  Tsutsumi-Hatano-1, Tsutsumi-Hatano-2} the reader can find results about well-posedness, ill-posedness and existence/stability of solitary waves. 

\subsection{Goal and motivation}

In the one-dimensional case all the systems above presented are globally well-posed in the natural energy space. Our interest here  is to study the behavior of solutions of the Cauchy problem associated to \eqref{S-GBenney}, or equivalently of \eqref{S-Benney}, in the energy space on any time interval $\Delta T:=[0, T]$ when $\varepsilon\to0$. Formally, when $\varepsilon\to0$ the system \eqref{S-GBenney} decouples and the solutions $(u_{\varepsilon},v_{\varepsilon}, w_{\varepsilon})$ are reduced to satisfy 
\begin{equation}\label{NLS}
\begin{cases}
i\partial_tu+ \partial_x^2u = \big(\tau+\tfrac{\alpha\beta}{\lambda} + \tfrac{\alpha'\beta'}{\lambda'}\big)u\vert u\vert^2,&  (x,t)\in \R\times (0, T],\medskip \\
v=\tfrac{\beta}{\lambda}|u|^2,&  (x,t)\in\R \times [0, T], \\
z=\tfrac{\beta'}{\lambda'}|u|^2,&  (x,t)\in\R \times [0, T],
\end{cases}
\end{equation}
and in this case the limit is named \textit{adiabatic} or \textit{subsonic} because of the regime. The first component reaches to solve a cubic-NLS equation while the others components manage the quadratic nonlinearity. 

\medskip
In the case of \eqref{Zred} we have
\begin{equation}
v=n_{+}=\vert u\vert^2, \quad z=n_{-}=-\vert u\vert^2,
\end{equation}
with
\begin{equation}
n=\frac{1}{2}(n_{-}-n_{+})=-\vert u\vert^2,
\end{equation}
and therefore the focusing cubic-NLS equation limit
\begin{equation}
i\partial_tu+ \partial_x^2u = -u\vert u\vert^2.
\end{equation}

Essentially we will have to deal with the system \eqref{S-Benney} because $v$ and $w$ are not coupled in \eqref{S-GBenney}, and in this case we have the limit system
\begin{equation}\label{cNLS}
\begin{cases}
i\partial_tu+ \partial_x^2u = \frac{\alpha\beta}{\lambda}u|u|^2,& (x,t)\in \R\times (0, T],\medskip \\
v=\tfrac{\beta}{\lambda}|u|^2,& (x,t)\in\R\times [0, T],
\end{cases}
\end{equation} 
whose solution $u$ with initial data $u_0$ verifies
\begin{equation}
u(x,t)=S(t)u_0-i\tfrac{\alpha \beta}{\lambda}\int_0^{t}S(t-s)\vert u(x,s)\vert^2 u(x,s)ds,
\end{equation}
where
\begin{equation}\label{Schrodinger-Group}
S(t)=e^{it\partial_x^2}
\end{equation}
denotes the unitary group associated to the linear Schr\"odinger equation.

\medskip 
Clearer evidence that this convergence is possible comes from the convergence of the traveling waves of the Benney system to those of equation \eqref{cNLS}. More precisely, for a given $c>0$ and $\alpha, \beta, \lambda$ satisfying
\begin{equation}\label{cond-waves}
\frac{\alpha \beta}{\lambda} <0,
\end{equation}
the traveling waves
\begin{equation}\label{traveling-waves-Benney}
\begin{cases}
u_{c, w, \varepsilon}(x,t)=e^{iwt}e^{\frac{ic}{2}(x-ct)}\sqrt{\frac{2(\varepsilon c -\lambda)}{\alpha \beta}}\,
\sigma\, \text{sech}\,(\sigma (x-ct)), \medskip\\
v_{c, w,\varepsilon}(x,t)=-\frac{2}{\alpha}\sigma^2 \text{sech}^2(\sigma (x-ct)),
\end{cases}
\end{equation}
with $0< \varepsilon < 1/c$ and $\sigma = \sqrt{w-c^2/4}$ are solutions of the Benney system and, on the other hand,
under condition \eqref{cond-waves} the family 
\begin{equation}\label{traveling-waves-CNLS}
u_{c, w}(x,t)=e^{iwt}e^{\frac{ic}{2}(x-ct)}\sqrt{\tfrac{-2\lambda}{\alpha \beta}}\,
\sigma\, \text{sech}\,(\sigma (x-ct)),
\end{equation}
describes solutions of \eqref{cNLS}. Then, by using  the Lebesgue's dominated convergence theorem we can verify that
\begin{equation}\label{convergence-waves}
\lim\limits_{\varepsilon \to 0}\big\|u_{c, w, \varepsilon} - u_{c, w}\big\|_{L^{\infty}_tL^2_x}=0,
\end{equation}
moreover 
\begin{equation}
v_{c, w, \varepsilon}(x,t)= \frac{\beta}{\lambda} |u_{c, w}(x,t)|^2,
\end{equation}
for all $\varepsilon >0$.

\medskip 
Traveling waves are global localized solutions of system \eqref{S-Benney}  belonging to the energy space $H^1\times L^2$. However, for any initial data $(u_0, v_0) \in H^1\times L^2$ the corresponding solution $(u_\varepsilon,v_\varepsilon)$ of \eqref{S-Benney} is global in time, as we will described in Section \ref{Benney-section} below. So, it's natural to ask whether the convergence $(u_\varepsilon,v_\varepsilon)\to \big(u,\frac{\beta}{\lambda}\vert u\vert^2\big)$ holds in $H^1\times L^2$ or at least in $L^2\times L^2$. 

\begin{remark}
We highlight that the limit to the cubic-NLS showed in \cite{Added-Added, Schochet-Weinstein, Oliveira2} is in the sense of punctual convergence in time. Moreover, as in  \cite{Ozawa-Tsutsumi, KPV}, the technique used in the proof demanded more than four derivatives on the initial data. In \cite{Added-Added} the authors also proved weak convergence by using compactness arguments on the energy space  and in \cite{Oliveira2} it was used the theory of symmetric hyperbolic system to deal with the limit of \eqref{ZR} but the result of this work is weaker.
\end{remark}

Our goal in this work is to improve, in some sense in one dimension, the known convergence results  for these systems, taking as starting point the study of the convergence for solutions of the  \eqref{S-Benney}. The key point in \eqref{S-Benney} is that we will take advantage of the transport phenomenon in a single equation in order to find a rate for the convergence of solutions to the respective solution of \eqref{cNLS} in the space $\cC\big([0,T]; \,L^2\times L^2\big)$. Then we replicate the same argument to get a rate for the convergence of solutions of the more general system \eqref{S-GBenney} to the solution of the equation \eqref{NLS} in the space $\cC\big([0,T];\,L^2\times L^2\times L^2\big)$. Consequently,  we derive rates for the convergence of solutions of the systems \eqref{1dZR} and \eqref{Zred}, and also for system \eqref{Z} in the respective topology induced by the energy space. In this sense our results are stronger. 

\medskip 
Notice that in \cite{Schochet-Weinstein}, the real-valued functions of the original system behaves like quadratic nonlinearity of the same system as $\varepsilon$ tends to 0, but this is not met in the limit. We also can prove that this is really achieved too, with less regular initial data, as long as the solutions stay uniformly bounded respect to $\varepsilon$. Therefore, our results are a significant improvement, which will be established in the next section.

\section{Main results}
It will be important to distinguish solutions corresponding to different values of $\varepsilon$ for system \eqref{S-GBenney}, which justifies the notation $u_{\varepsilon}=u_{\varepsilon}(x, t), v_{\varepsilon}=v_{\varepsilon}(x, t)$ and $z_{\varepsilon}=z_{\varepsilon}(x, t)$ whenever this is necessary.

\medskip 
\begin{theorem}\label{Head-Th}
Let $T>0$ be given. Suppose that the solutions $(u_{\varepsilon}, v_{\varepsilon}, z_{\varepsilon})$ of system \eqref{S-GBenney} corresponding to a family of initial data
$\big\{({u_0}_{\varepsilon}, {v_0}_{\varepsilon}, {z_0}_{\varepsilon})\big\}_{0< \varepsilon <1} \in H^1\times L^2\times L^2,$
satisfy
\begin{equation}\label{bounded-u-epsilon}
\sup_{0< \varepsilon < 1}(\|{u_{\varepsilon}}\|_{L^{\infty}_TH^1_x}+\|v_{\varepsilon}\|_{L^{\infty}_TL^2_x}+\|z_{\varepsilon}\|_{L^{\infty}_TL^2_x})<\infty
\end{equation}
and
\begin{equation}
\vert u_\varepsilon\vert\leq\psi \quad a.e,
\end{equation}
for some $\psi\in L^\infty([0,T];\,L^\infty\cap L^2)$, and  let $u$ the solution of the cubic nonlinear Schr\"odinger equation in \eqref{NLS} with data $u_0\in H^1$. Then, 
\begin{equation}
\lim_{\varepsilon\to0}\big(\| u_\varepsilon-u\|_{L^\infty_TL^2_x}+\|v_\varepsilon-\tfrac{\beta}{\lambda}\vert u\vert^2\|_{L^\infty_TL^2_x}+\|z_\varepsilon-\tfrac{\beta'}{\lambda'}\vert u\vert^2\|_{L^\infty_TL^2_x}\big)=0,
\end{equation}
whenever the  family of initial data  satisfies 
\begin{equation}\label{Cconditions}
\lim_{\varepsilon\to0}\big(\|{u_0}_{\varepsilon}-u_0\|_{H^1}+\|{v_0}_{\varepsilon}-\tfrac{\beta}{\lambda}\vert u_0\vert^2\|_{L^2}+ \|{z_0}_{\varepsilon}-\tfrac{\beta'}{\lambda'}\vert u_0\vert^2\|_{L^2}\big)=0.
\end{equation}
\end{theorem}

\begin{theorem}\label{Th1}
Let $T>0$ be given and  consider $\frac{\alpha\lambda}{\beta}<0$. Suppose that  $\big\{({u_0}_{\varepsilon}, {v_0}_{\varepsilon})\big\}_{0< \varepsilon <1}$ is a family of data in the space $H^1\times L^2$  such that
\begin{equation}\label{apriori-estimate-H1L-hypotheses-a}
\sup_{0< \varepsilon < 1}(\|{u_0}_{\varepsilon}\|_{H^1}+\|{v_0}_{\varepsilon}\|_{L^2})<\infty
\end{equation}
and 
\begin{equation}\label{Comp}
\lim\limits_{\varepsilon \to 0}\big\|{v_0}_{\varepsilon}-\tfrac{\beta}{\lambda}\vert {u_0}_{\varepsilon } \vert^2\big\|_{L^2}=0.
\end{equation}
Then, the corresponding solutions $(u_{\varepsilon}, v_{\varepsilon})$ of the Cauchy problem associated to the  \eqref{S-Benney} in $H^1\times L^2$ provided in \cite{G-T-V} satisfies
\begin{equation}\label{Th1-conclusion-a}
\big\| v_\varepsilon-\tfrac{\beta}{\lambda}\vert u_\epsilon\vert^2 \big\|_{L_T^\infty L_x^2}=O(\varepsilon)
\end{equation}
and
\begin{equation}\label{Th1-conclusion-b}
\Big\|\int_0^tS(t-s)[u_\varepsilon v_\varepsilon-\tfrac{\beta}{\lambda}u_\varepsilon \vert u_\varepsilon\vert^2](x,s)ds \Big\|_{L_T^\infty L_x^2}\lesssim T^{3/4}O(\varepsilon).
\end{equation}
Furthermore, the same is true in the case $\frac{\alpha\lambda}{\beta}>0$ under the extra hypothesis $\frac{\beta}{\lambda^2}=O(\epsilon^3)$.
\end{theorem}    

\begin{corollary}\label{Th1-corollary}
Assume the hypotheses in Theorem \ref{Th1} in the case $\frac{\alpha\lambda}{\beta}<0$. Suppose further that
\begin{equation}
\vert u_\varepsilon\vert\leq \psi\quad a.e
\end{equation}
for some function $\psi\in L^\infty([0,T];\,L^\infty\cap L^2)$. 
If $u$ is the solution of the cubic nonlinear Schr\"odinger equation \eqref{cNLS} with initial data $u_0\in H^1$, then
\begin{equation}
\lim_{\varepsilon\to0}\big\| v_\varepsilon-\tfrac{\beta}{\lambda}\vert u\vert^2 \big\|_{L_T^\infty L_x^2}=0
\end{equation}
and
\begin{equation}\label{u-varepsilon-NLS-estimate-b}
\lim_{\varepsilon\to0}\|u_{\varepsilon}-u\|_{L^{\infty}_TL^2_x}=0,
\end{equation}
whenever $\Vert {u_0}_{\varepsilon}-u_0\Vert_{L^2}\to0$.
\end{corollary}

\begin{theorem}\label{Th2}
Consider the Cauchy problem associated to \eqref{S-Benney} with $\frac{\alpha\lambda}{\beta}<0$. If $\big\{({u_0}_{\varepsilon}, {v_0}_{\varepsilon})\big\}_{0< \varepsilon <1}$ is a family of data in the space $H^1\times L^2$  such that
\begin{equation}\label{hypotheses}
\lim_{\varepsilon\to0}(\|{u_0}_{\varepsilon}-u_0\|_{H^1}+\|{v_0}_{\varepsilon}-\tfrac{\beta}{\lambda}\vert u_0\vert^2\|_{H^1})=0
\end{equation}
and $u$ is the solution of \eqref{cNLS} with initial data $u_0$, then we have
\begin{equation}\label{Th2-a}
\Big\|\partial_x\int_0^tS(t-s)[u_\varepsilon v_\varepsilon-\tfrac{\beta}{\lambda}u_\varepsilon \vert u_\varepsilon\vert^2](x,s)ds \Big\|_{L_T^\infty L_x^2} \lesssim  T^{3/4}O(\varepsilon).
\end{equation} 
\end{theorem}

\begin{remark}
It is important to note that:
\begin{enumerate}[(a)]
\item The inequalities \eqref{Th1-conclusion-b} and \eqref{Th2-a} means, respectively,  that
\begin{equation}\label{u-varepsilon-NLS-estimate-L2}
\Big\|u_{\varepsilon}-S(t)u_{0\epsilon}+ i\tfrac{\alpha \beta}{\lambda}\int_0^tS(t-s)u_{\varepsilon}|u_{\varepsilon}|^2ds\Big\|_{L^{\infty}_TL^2_x}\lesssim T^{3/4}O(\varepsilon).
\end{equation}
and 
\begin{equation}\label{u-varepsilon-NLS-estimate-H1}
\Big\|\p_x\Big(u_{\varepsilon}-S(t)u_{0\epsilon}+ i\tfrac{\alpha \beta}{\lambda}\int_0^tS(t-s)u_{\varepsilon}|u_{\varepsilon}|^2ds\Big)\Big\|_{L^{\infty}_TL^2_x}\lesssim T^{3/4}O(\varepsilon).
\end{equation}

\medskip 
\item The sign $\frac{\alpha\lambda}{\beta}<0$ is because it enables to obtain bonded solutions respect to $\varepsilon$. 

\medskip 
\item For the 1d-Zakharov system we have removed the condition $n_1\in\dot H^{-1}$ whenever we have  uniformly bounded solutions on the parameter $\varepsilon$. Also we do not impose decay for high frequency on initial data. 
\end{enumerate}
\end{remark}

\section{Preliminary results}

\subsection{On the 1d-Zakharov-Rubenchik}
As we commented in the introduction, the best known result about local well-posedness for the Zakharov-Rubenchik system \eqref{ZR} was established in  \cite{Linares-Matheus}. The flux of this system preserves the following nonlinear quantities: 
\begin{equation}
\mathcal{I}_1(t)=\int_{\mathbb{R}}\vert u(x,t)\vert^2dx,
\end{equation}

\begin{multline}
\mathcal{I}_2(t)=\frac{1}{2}\int_{\mathbb{R}}\vert \partial_xu(x,t)\vert^2dx+\frac{ck}{4}\int_{\mathbb{R}}\vert u(x,t)\vert^4dx+\frac{k}{2}\int_{\mathbb{R}}\Big(\varphi-\frac{a}{2}\rho\Big)(x,t)\vert u(x,t)\vert^2dx\\
+\frac{b}{4}\int_{\mathbb{R}}\vert \rho(x,t)\vert^2dx+\frac{1}{4}\int_{\mathbb{R}}\vert \varphi(x,t)\vert^2dx-\frac{b}{2}\int_{\mathbb{R}}\rho(x,t)\varphi(x,t)dx,
\end{multline}

\begin{equation}
\mathcal{I}_3(t)=\varepsilon\int_{\mathbb{R}}\rho(x,t)\varphi(x,t)dx+\frac{i\varepsilon}{2}\int_{\mathbb{R}}(u(x,t)\partial_x\bar u(x,t)-\bar u(x,t)\partial_x u(x,t))dx
\end{equation}
and consequently
\begin{equation}
\mathcal{I}_4(t)=\mathcal{I}_2(t)+\frac{b}{2\varepsilon}\mathcal{I}_3(t).
\end{equation}

\medskip 
The conservation laws above yield the following global well-posedness:

\begin{theorem}[\cite{Linares-Matheus}]
The Cauchy problem associated to the system \eqref{1dZR} is globally well posed for any initial data $(u_0,\psi_{10},\psi_{20})$ belonging to the spaces:
\begin{enumerate}
	\item[(a)] $H^{s+1/2}\times H^s\times H^s$ with $s\ge 0$, \medskip 
	\item[(b)] $H^1\times L^2\times L^2$ whenever $b>a^2$.
\end{enumerate}
Furthermore, in the energy space the solutions satisfied the uniform control 
\begin{equation}
\Vert (u(t),\psi_1(t),\psi_2(t))\Vert_{H^1\times L^2\times L^2}\lesssim \Vert (u_0,\psi_{10},\psi_{20})\Vert^2_{H^1\times L^2\times L^2}+\Vert u_0\Vert^6_{L^2},
\end{equation}
for all $t\ge 0$. 
\end{theorem}
 
For more details on this matter we recommend \cite{Linares-Matheus,Oliveira1}. Also notice that if the initial data in the Zakahrov-Rubenchik system (depending on $\varepsilon$) satisfy the compatibility conditions \eqref{Cconditions} then the solutions are bounded uniformly.     
 
\subsection{On the 1d-Zakharov system} 
Concerning the existence of local  solutions for the Cauchy problem associated to \eqref{Z}, a wide class of regularity was obtained in \cite{G-T-V}
for data 
$$(u_0, n_0, n_1)\in H^s\times H^{\kappa}\times H^{\kappa-1},\quad \text{with} \quad n_1=n_t(0),$$
in the case $\varepsilon =1$. We  notice that the  same theory holds for all 
$\varepsilon >0$ because of the re-scaling:
$$u(x,t)=U(\varepsilon^{-1} x,\varepsilon^{-2}t)\quad \text{and}\quad n(x,t)=N(\varepsilon^{-1} x,\varepsilon^{-2}t),$$ 
which transform the system \eqref{Z} to 
\begin{equation}\label{Z-GTV}
	\begin{cases}
		i\partial_{r}U+\partial_{z}^2u=\varepsilon^2 N U,\medskip \\
		\partial_{r}^2N -\partial_{z}^2N=\partial_x^2\vert U\vert^2,
	\end{cases}
\end{equation} 
with $r= \varepsilon^{-2}t$ and $z=\varepsilon^{-1}x$. Indeed, system \eqref{Z-GTV} is included in the local theory developed in \cite{G-T-V}, since bi-linear estimates using for the authors do not depend on the coefficients of nonlinear terms. More precisely, we have the following result:

\medskip 
\begin{theorem}[\cite{G-T-V}]\label{lwp-Z-1d}
	For any $(u_0, n_0, n_1) \in H^s(\R)\times H^{\kappa}(\R)\times H^{\kappa -1}$, with $s$ and $\kappa$ verifying the conditions:
	\begin{equation}\label{wp-condition-Z-1d}
	-\frac{1}{2}< s-\kappa \le 1\quad \text{and}\quad -\frac{1}{2}\le \kappa \le 2s-\frac{1}{2},
	\end{equation}
	there exists a positive time $T=T(\|u_0\|_{H^s}, \|n_0\|_{H^{\kappa}}, \|n_1\|_{H^{\kappa-1}})$ and a unique solution $(u(t,\cdot),v(t, \cdot))$ of the initial value problem
	\eqref{Z} in the time interval $[0,T]$, satisfying
	$$(u, n, n_t)\in \cC\left([0,T]; H^s(\R) \times H^{\kappa}(\R) \times H^{\kappa -1}(\R)\right).$$
	Moreover, the map $(u_0,n_0, n_1) \longmapsto (u(\cdot, t),n(\cdot, t), n_t(\cdot, t))$ is locally Lipschitz.
\end{theorem}

The Figure  \ref{WP-Region-GTV} shows the region $\mathcal{W}$ of the Sobolev indexes defined by conditions \eqref{wp-condition-Z-1d}.

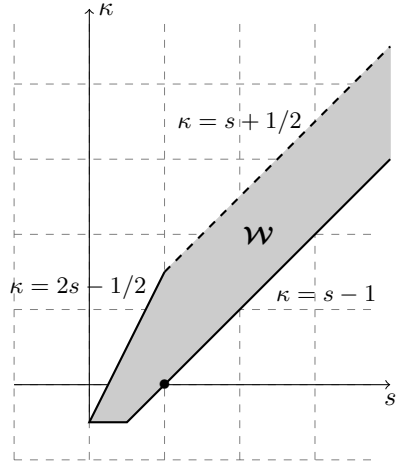
\begin{figure}[htp]
	\centering
	\begin{tikzpicture}
	\draw[very thin,color=gray, dashed] (-1,-1) grid (3.8,4.8);
	\draw[->] (-1,0)--(4,0) node[below] {$s$};
	\draw[->] (0,-0.5)--(0,5) node[right] {$\kappa$};
	\filldraw[color=gray!40](0,-0.5)--(0.5,-0.5)--(4,3)--(4,4.5)--(1,1.5)--(0,-0.5);
	\draw[thick](1,1.5)--(0,-0.5)--(0.5,-0.5)--(4,3);
	\draw[thick, dashed](1,1.5)--(4,4.5);
	\node at (2.25,2){\small{$\boldsymbol{\mathcal{W}}$}};
	\node at (-0.15,1.3){\small{$\kappa=2s-1/2$}};
	\node at (2,3.5){\small{$\kappa=s+1/2$}};
	\node at (3.15,1.2){\small{$\kappa=s-1$}};
	\node at (1,0){\small{$\bullet$}};
	\end{tikzpicture}
	\caption{Local well-posedness  regularity for \eqref{Z} established in \cite{G-T-V}.}
	\label{WP-Region-GTV}
\end{figure}

As we see, energy regularity $H^1\times L^2\times H^{-1}$ is covered in Theorem \ref{lwp-Z-1d}, so global well-posedness in this space is automatically provided by
conservation laws:

\begin{equation}
\mathcal{J}_1(t)=\int_{\R} |u(x,t)|^2dx
\end{equation}
and
\begin{equation}
\mathcal{J}_2(t)=\int_{-\infty}^{+\infty}\Big( \vert\partial_xu(x,t)\vert^2+n(x,t)\vert u(x,t)\vert^2+\frac{1}{2}n^2(x,t)+\frac{\varepsilon^2}{2}\vert v(x,t)\vert^2 \Big)dx,
\end{equation}
due to the Hamiltonian version of the Zakharov system \eqref{Z}:
\begin{equation}\label{HamZ}
\begin{cases}
i\partial_tu+\partial_{xx}u=nu,\\
\partial_{t}n+\partial_x v=0,\\
\varepsilon^2\partial_t v+\partial_x n=-\partial_x\vert u \vert^2,
\end{cases}
\end{equation} 
where $v(x,t)=\partial_x w(x,t)$ and $\partial_tn(x,0)=n_1(x)=-\partial_{xx} w_0$ for a certain function $w$.

\medskip
We can notice that
\begin{equation}
v=\partial^{-1}_x\partial_tn=\partial^{-1}_xn_1\quad\text{in $t=0$},
\end{equation}
so $\mathcal{J}_2$ can be written only in terms of $u$ and $n$.

\medskip
Finally, as solutions of \eqref{Zred} are solutions of \eqref{Z}, where  
\begin{equation}
\partial_xn_{\mp}=\varepsilon\partial_tn\pm\partial_xn,
\end{equation}
then
\begin{equation}
\partial_xn_{\mp}=-\varepsilon\partial_xv\pm\partial_xn=\partial_x(-\varepsilon v\pm n);
\end{equation}
so
\begin{equation}
n_{\mp}=-\varepsilon v\pm n
\end{equation}
and
\begin{equation}
\Vert n_{\mp}\Vert_{L^2}\leq\varepsilon \Vert v\Vert_{L^2}+\Vert n\Vert_{L^2}.
\end{equation}

\medskip
This indicates a control of the components $n_{\mp}$ in function of the original variables $v$ and $n$, whence the solutions $(u_\varepsilon, n_{\varepsilon_{\pm}})$ are uniformly bounded in $H^1\times L^2$. However, as we will see later, the system \eqref{S-GBenney} has an intrinsically very good structure with three conserved quantities, so we do not need to assume $n_1\in \dot H^{-1}$ as in \eqref{HamZ} to have a control of the solutions on $\varepsilon$.
  
\subsection{On the Benney system}\label{Benney-section}
We will denote by $\cT_{\varepsilon}(t)$ the family of translator operators associated to the free wave equation
\begin{equation}\label{wave-linear}
\begin{cases}
v_t + \frac{\lambda}{\varepsilon}v_x =0,\\
v(x,0)=v_0,
\end{cases}
\end{equation}
that is, 
\begin{equation}\label{Translator-Operator}
\cT_{\varepsilon}(t)v_0 = v_0\big(x-\tfrac{\lambda}{\varepsilon}t\big).
\end{equation}

\medskip 
The most general theory concerning local well-posedness, known so far, in Sobolev spaces for the Cauchy problem associated to the Benney system \eqref{S-Benney} also was derived in \cite{G-T-V}, where the authors established local well-podness for initial data $(u_0, v_0)\in H^s(\R)\times H^{\kappa}(\R)$ in the same region of regularity  showed in Figure \ref{WP-Region-GTV},
with time of the existence $T$, depending on the norms $\|u_0\|_{H^s},$ and  $\|v_0\|_{H^{\kappa}}$. Indeed, the results were obtained as a corollary of the proof of Theorem \ref{lwp-Z-1d}. 

\medskip 
The solution for the system \eqref{S-Benney} with initial data $({u_0}_{\varepsilon}, {v_0}_{\varepsilon})$ satisfy the following integral equations:
\begin{equation}\label{DuhamelEq}
\begin{cases}
\displaystyle u_{\varepsilon}(x,t)=S(t)u_{0\varepsilon}-i\alpha\int_0^{t}S(t-s)u_{\varepsilon}(x,s)v_{\varepsilon}(x,s)ds,\medskip\\
\displaystyle v_{\varepsilon}(x,t)=\cT_{\varepsilon}(t)v_{0\varepsilon}+\tfrac{\beta}{\varepsilon}\int_0^{t}\cT_{\varepsilon}(t-s)\partial_x\vert u_{\varepsilon}(x,s)\vert^2ds.
\end{cases}
\end{equation}

\medskip 
The flow of the system \eqref{S-Benney} preserves the following nonlinear functional:

\begin{equation}\label{Mass}
\cM_{\varepsilon}(t):=\int_{-\infty}^{+\infty}|u_{\varepsilon}(x,t)|^{2}dx=\cM_{\varepsilon}(0)\quad \big(\text{mass}\big),
\end{equation}
 
\begin{equation}\label{Momment}
\cK_{\varepsilon}(t):=\int_{-\infty}^{+\infty}\Big(|v_{\varepsilon}(x,t)|^2 +\tfrac{2\beta}{\alpha \varepsilon} \text{Im}\, u_{\varepsilon}(x,t)\p_x\bar{u}_{\varepsilon}(x,t)\Big)dx=\cK_{\varepsilon}(0)
\quad \big(\text{moment}\big)
\end{equation}
and
\begin{equation}\label{Energy}
\cE_{\varepsilon}(t):=\int_{-\infty}^{+\infty}\Big(|\p_xu_{\varepsilon}(x,t)|^2 +\alpha v_{\varepsilon}(x,t)|u_{\varepsilon}(x,t)|^2 
-\tfrac{\alpha \lambda}{2\beta}v^2_{\varepsilon}(x,t)\Big)dx=\cE_{\varepsilon}(0),
\quad \big(\text{energy}\big),
\end{equation}
for all $0\le t< T^*_{\varepsilon}$, where $T^*_{\varepsilon}$ is the maximal time of existence for the respective solution. 

\medskip 
Our  main goal is to study the asymptotic behavior for  solutions of the Benney system  in the topology of $\cC\big([0,T];\,H^1(\mathbb{R})\big)$ for the component  $u_{\varepsilon}$ and with an appropriated topology for the corresponding transport solution  $v_{\varepsilon}$. However, if we expect a strong convergence result in $\cC([0,T];L^2(\mathbb{R}))$ for solutions  $v_{\varepsilon}$ it is natural to impose a compatibility condition on the initial data, like

\begin{equation}
v_0=\frac{\beta}{\lambda}\vert u_0 \vert^2
\end{equation}
or
\begin{equation}
\lim\limits_{\varepsilon \to 0}\big\|\lambda {v_0}_{\varepsilon}-\beta\vert {u_0}_{\varepsilon } \vert^2\big\|_{L^2}=0 
\end{equation}
if the data vary with $\varepsilon$. For instance, if  $u_{0\varepsilon}\equiv0$ for all $\varepsilon$ then 
\begin{equation}
u_{\varepsilon}(x, t)\equiv0\quad\text{and}\quad v_{\varepsilon}(x,t)=v_{0\varepsilon}\big(x-\tfrac{\lambda}{\varepsilon}t\big),
\end{equation}
so  
\begin{equation}
\Vert v_{\varepsilon}(x,t) \Vert_{L_t^{\infty}L^2_x}=\Vert v_{0\varepsilon}\Vert_{L_x^2}.
\end{equation}
Hence,  we do not have much chance of show convergence in the space $L_T^{\infty}H^1_x\times L_T^{\infty}L^2_x$ without assu\-ming that
$\Vert v_{0\varepsilon}\Vert_{L^2_x} \to 0$ as $\varepsilon\to0$. For non compatible initial data an initial layer phenomenon should appear.

\subsection{Strichartz estimates}
Finally, we recall  some smoothing effects for the one-dimensional free Schr\"odinger  group $\displaystyle S(t)$.

\begin{lemma}[\textbf{Strichartz estimates} \cite{Cazenave-Book}] Let $(p_1,q_1)$ and $(p_2,q_2)$ be two pairs of admissible exponents for $S(t)$ in $\R$; that is, both satisfying
	the condition
	\begin{equation}\label{Strichart-Admissible}
	\frac{2}{p_i}=\frac12 -\frac 1{q_i}\quad \text{and} \quad 2\le q_i\le \infty\quad (i=1,2).
	\end{equation}
	Then, for any $0<T\leq \infty$, we have
	\begin{equation}\label{Strichart-homogeneous}
	\|S(t)f\|_{L^{p_1}_TL^{q_1}_x}\le c\|f\|_{L^2(\R)},
	\end{equation}
	as well as the non-homogeneous version
	\begin{equation}\label{Strichart-non-homogeneous}
	\left\|\int_0^tS(t-s)g(\cdot, s)ds\right\|_{L^{p_1}_TL_x^{q_1}}\le c \|g\|_{L^{p'_2}_TL_x^{q'_2}},
	\end{equation}
	where $1/p_2+1/p_2'=1$, $1/q_2 + 1/q'_2=1$. The constants in both inequalities are independent of $T$. 
\end{lemma}

\section*{\textbf{Acknowledgments}}
A. J. Corcho  would like to thank the support given by the Graduate Program in Mathematics of the Universidade Federal do Rio de Janeiro - UFRJ. J. C. Cordero would like to thank to the Instituto de Matem\'atica at UFRJ and to the Instituto de Matem\'atica Pura e Aplicada - IMPA for the support during the Postdoctoral Summer Program 2020, where part of this work was done. We would like to thank Hermano Frid for suggesting us to study the dynamics of the Benney system.  We are also grateful to Felipe Linares for some useful comments on a previous version.    

\section{Energy estimates and weak convergence}

In this section we present some estimates that will be useful to prove the statements of the main results. We describe the dynamic of the global solutions of \eqref{S-Benney} in the space $H^1\times L^2$ with respect to the parameter $\varepsilon$, when extra hypotheses are put on the initial data. 

\subsection{A priori estimates for the Benney system}
\begin{lemma}\label{apriori-estimate-H1L2}
If $\big\{({u_0}_{\varepsilon}, {v_0}_{\varepsilon})\big\}_{0< \varepsilon <1}$ is a family of data in the space $H^1\times L^2$  such that
\begin{equation}\label{apriori-estimate-H1L-hypotheses}
\sup_{0< \varepsilon < 1}(\|{u_0}_{\varepsilon}\|_{H^1}+\|{v_0}_{\varepsilon}\|_{L^2})<\infty,
\end{equation}
then the corresponding solutions $(u_{\varepsilon}, v_{\varepsilon})$ of the IVP \eqref{S-Benney} in $H^1\times L^2$ provided by Theorem A  satisfies
\begin{align}
&\sup_{0< \varepsilon < 1}(\|{u_{\varepsilon}}\|_{L^{\infty}_tH^1_x}+\|v_{\varepsilon}\|_{L^{\infty}_tL^2_x})<\infty \quad \text{if}\quad \tfrac{\alpha \lambda }{\beta}<0,\label{grow-norm-A}\\
&\|{u_{\varepsilon}}\|_{L^{\infty}_tH^1_x}+\|v_{\varepsilon}\|_{L^{\infty}_tL^2_x}=O(1/\varepsilon)\quad \text{if}\quad \tfrac{\alpha \lambda }{\beta}>0.\label{grow-norm-B}
\end{align}
\end{lemma}
\begin{proof}
We begin with the proof of \eqref{grow-norm-A}. From \eqref{Energy} we have 
\begin{equation*}\label{proof-grow-1a}
\begin{split}
\|\p_xu_{\varepsilon}\|_{L^2}^2 + |\tfrac{\alpha}{2\beta}|\,\|v_{\varepsilon}\|_{L^2}^2&=\cE_{\varepsilon}(0) - \alpha \int_{-\infty}^{+\infty} v_{\varepsilon}|u_{\varepsilon}|^2dx\\
&\le \cE_{\varepsilon}(0) + |\alpha|\,\|v_{\varepsilon}\|_{L^2}\|u_{\varepsilon}\|_{L^4}^2\\
&\le \cE_{\varepsilon}(0) + |\alpha| \big( \tfrac{1}{4|\beta|} \|v_{\varepsilon}\|_{L^2}^2 + |\beta|\, \|u_{\varepsilon}\|^4_{L^4}\big).
\end{split}
\end{equation*}
So, by using Gagliardo-Nirenberg inequality and \eqref{Mass} we have
\begin{equation*}\label{proof-grow-1b}
\|\p_xu_{\varepsilon}\|_{L^2}^2 + |\tfrac{\alpha}{4\beta}| \|v_{\varepsilon}\|_{L^2}^2
\le \cE_{\varepsilon}(0)+  c\cM_{\varepsilon}^{3/2}(0)\,\|\p_xu_{\varepsilon}\|_{L^2},
\end{equation*}
that allows us to conclude 
\begin{equation*}\label{proof-grow-1c}
\|\p_xu_{\varepsilon}(\cdot, t)\|_{L^2}^2 + \|v_{\varepsilon}(\cdot, t)\|_{L^2}^2 \lesssim_{\alpha, \beta} \; \cE_{\varepsilon}(0) + \cM_{\varepsilon}^3(0), \quad \text{for all}\;  t\ge 0.
\end{equation*}
Then, from \eqref{apriori-estimate-H1L-hypotheses} we deduce immediately  \eqref{grow-norm-A}. 

\medskip 
Now we proceeds with the proof of  \eqref{grow-norm-B}. Again, using \eqref{Energy} and Gagliardo-Nirenberg inequality we obtain
\begin{equation}\label{proof-grow-2a}
\begin{split}
\|\p_xu_{\varepsilon}\|_{L^2}^2 &=\cE_{\varepsilon}(0) + |\tfrac{\alpha}{2\beta}|\,\|v_{\varepsilon}\|_{L^2}^2 - \alpha \int_{-\infty}^{+\infty} v_{\varepsilon}|u_{\varepsilon}|^2dx\\
& \le \cE_{\varepsilon}(0) + \big(|\tfrac{\alpha}{2\beta}| + |\tfrac{\alpha}{2}| \big)\,\|v_{\varepsilon}\|_{L^2}^2 +  |\tfrac{\alpha}{2}|\|u_{\varepsilon}\|^4_{L^4}\\
&\le \cE_{\varepsilon}(0) + \big(|\tfrac{\alpha}{2\beta}| + |\tfrac{\alpha}{2}| \big)\,\|v_{\varepsilon}\|_{L^2}^2 +  
c|\tfrac{\alpha}{2}|\cM_{\varepsilon}^{3/2}(0)\|\partial_xu_{\varepsilon}\|_{L^2}\\
&\le \cE_{\varepsilon}(0) + \big(|\tfrac{\alpha}{2\beta}| + |\tfrac{\alpha}{2}| \big)\,\|v_{\varepsilon}\|_{L^2}^2 +  
c_{\alpha, \beta}\cM_{\varepsilon}^{3}+\tfrac{1}{4}\|\partial_xu_{\varepsilon}\|^2_{L^2}.
\end{split}
\end{equation}
On the other hand, from \eqref{Momment} we get 
\begin{equation}\label{proof-grow-2b}
\begin{split}
\|v_{\varepsilon}(\cdot, t)\|_{L^2}^2 &\le \cK_{\varepsilon}(0) -\tfrac{2\beta}{\alpha \varepsilon}\text{Im}\int_{-\infty}^{+\infty} u_{\varepsilon} \partial_x\bar{u}_{\varepsilon}dx\\
&\le \cK_{\varepsilon}(0) + \big|\tfrac{2\beta}{\alpha\varepsilon}\big|\cM_{\varepsilon}^{1/2}(0)\|\partial_xu_{\varepsilon}\|_{L^2}\\
&\le  \cK_{\varepsilon}(0) + \big|\tfrac{2\beta^2}{\alpha^2\varepsilon^2 \delta}\big|\cM_{\varepsilon}(0) + \tfrac{\delta}{2}\|\partial_xu_{\varepsilon}\|_{L^2}^2,
\end{split}
\end{equation}
for any positive $\delta$. 

\medskip 
Then, taking a suitable $\delta=\delta(\alpha, \beta)$ and inserting \eqref{proof-grow-2b} in \eqref{proof-grow-2a} we have
\begin{equation}\label{proof-grow-2c}
\|\p_xu_{\varepsilon}(\cdot, t)\|_{L^2}^2 \lesssim_{\alpha, \beta} \; \cE_{\varepsilon}(0) + \cK_{\varepsilon}(0) + \cM_{\varepsilon}(0) + \cM^3_{\varepsilon}(0).
\end{equation}

\medskip 
Now we note that from \eqref{apriori-estimate-H1L-hypotheses} we get $\cK_{\varepsilon}(0) = O(\frac{1}{\varepsilon})$,  $\cM_{\varepsilon}(0)=O(1)$ and  
$\cE_{\varepsilon}(0)=O(1)$. Thus,  from \eqref{proof-grow-2b} and \eqref{proof-grow-2c} we obtain
\begin{equation}
\|\p_xu_{\varepsilon}(\cdot, t)\|_{L^2}^2  + \|v_{\varepsilon}(\cdot, t)\|_{L^2}^2 =  O(1/\varepsilon^2), 
\end{equation}
which implies \eqref{grow-norm-B}.
\end{proof}

\subsection{Weak limit for the Benney system}
Now we state a weak convergence theorem for solutions $(u_{\varepsilon}, v_{\varepsilon})$, namely
\begin{theorem}\label{Th:WC}
Let $(u_{\varepsilon}, v_{\varepsilon})$ be any solution of \eqref{S-Benney} with initial data satisfying the hypotheses of Lemma \ref{apriori-estimate-H1L2} and $\alpha\lambda/\beta<0$. There is $u\in L^\infty(\mathbb{R}_+;H^1)$ such that $u_\varepsilon\to u$ almost everywhere in $(x,t)\in\mathbb{R}\times (0,T)$ as $\varepsilon$ go to $0$, and $(u_{\varepsilon}, v_{\varepsilon})$ converges to $(u,\frac{\beta}{\lambda}\vert u\vert^2)$ in $L^{\infty}(\mathbb{R}_+;\, H^1)\times L^{\infty}(\mathbb{R_+};\, L^2)$ weak star, where $u=u(x,t)$ is the unique solution of the nonlinear Schr\"odinger equation 
\begin{equation}
i\partial_tu+\partial_x^2u=\tfrac{\alpha\beta}{\lambda}u\vert u\vert^2
\end{equation}
with data $u(x,0)=u_0(x)\in H^1$.
\end{theorem} 

We only will do a sketch of the proof because these argument are well known. More details on this technicality can be review in \cite{Added-Added, Cordero}, even for higher dimensional models of Schr\"odinger type.
\begin{proof}
Because of the uniform bounds given by Lemma \ref{apriori-estimate-H1L2}, we have a sequence $(u_\varepsilon, v_\varepsilon)$ and $(u,v)$ such that
\begin{equation}
\begin{cases}
u_\varepsilon  \overset{*}{\rightharpoonup} u&\text{in}\quad L^\infty(\mathbb{R}_+;H^1)\\
v_\varepsilon  \overset{*}{\rightharpoonup} v&\text{in}\quad L^\infty(\mathbb{R}_+;L^2)\\
\vert u_\varepsilon\vert^2  \overset{*}{\rightharpoonup} \Gamma &\text{in}\quad L^\infty(\mathbb{R}_+;L^2),
\end{cases}
\end{equation}
so 
\begin{equation}
\begin{cases}
\partial_{xx}u_\varepsilon  \overset{*}{\rightharpoonup} \partial_{xx}u &\text{in}\quad L^\infty(\mathbb{R}_+;H^{-1})\\
\partial_xv_\varepsilon  \overset{*}{\rightharpoonup} v&\text{in}\quad  L^\infty(\mathbb{R}_+;H^{-1})\\
\partial_x\vert u_\varepsilon\vert^2  \overset{*}{\rightharpoonup} \partial_x\Gamma &\text{in}\quad L^\infty(\mathbb{R}_+;H^{-1}).
\end{cases}
\end{equation}
As the map
\begin{align}\label{M:map}
H^1\times L^2&\to H^{-1} \\
(f,g)&\mapsto fg \notag
\end{align}
is continuous, one can assume that $u_\varepsilon v_\varepsilon$ has a weak$^*$ limit in $L^\infty(\mathbb{R}_+;H^{-1})$, namely
\begin{equation}
u_{\varepsilon}v_{\varepsilon}\overset{*}{\rightharpoonup}\varLambda \quad  \text{in} \quad L^{\infty}(\mathbb{R}_+; H^{-1}).
\end{equation}

Then
\begin{equation}
\begin{cases}
\partial_tu_\varepsilon  \overset{*}{\rightharpoonup} \partial_tu\quad\quad\text{in}\quad L^\infty(\mathbb{R}_+;H^{-1})\\
\partial_tv_\varepsilon  \overset{*}{\rightharpoonup} \partial_tv\quad\quad\text{in}\quad L^\infty(\mathbb{R}_+;H^{-1}),
\end{cases}
\end{equation}
and also
\begin{equation}
\begin{cases}
i\partial_tu+\partial_x^2u=\alpha\varLambda,\medskip \\
\lambda\partial_xv=\beta\partial_x\Gamma,
\end{cases}
\end{equation}
in the distribution sense, in $L^{\infty}(\mathbb{R}_+; H^{-1}).$

\medskip 
The proof of the result is finished when it is shown  that
\begin{equation}
\Gamma=|u|^2 \quad \text{and}\quad \varLambda=\frac{\beta}{\lambda}u\Gamma.
\end{equation}

\medskip 
To check this we consider the interval $[0,T]$, $\Omega\subset\mathbb{R}$ bounded, $B_0:=H^1(\Omega), B:=L^4(\Omega), B_1:=H^{-1}(\Omega)$ and the restriction $u_{\varepsilon}|_{\Omega}$. Next we use the Rellich-Kondrachov's theorem and the Lions-Aubin's theorem to have compact and continuous embeddings. Then some subsequence of  $u_{\varepsilon}|_{\Omega}$ (also labeled by $\varepsilon$) converges strongly to $u|_{\Omega}$ in $L^2([0,T]; L^4(\Omega))$. Hence
\begin{equation}
u_{\varepsilon}\underset{\varepsilon\to0}{\longrightarrow} u \ \text{strongly in} \  L^2([0,T]; L^2_{loc}(\mathbb{R})),
\end{equation}
and thus,
\begin{equation}
u_{\varepsilon}\underset{\varepsilon\to0}{\longrightarrow} u \quad a.e \quad  \text{in} \ (t,x)\in [0,T]\times \mathbb{R}
\end{equation}
and
\begin{equation}
|u_{\varepsilon}|^2\underset{\varepsilon\to0}{\longrightarrow}|u|^2 \quad  a.e. \quad  \text{in} \ (t,x)\in [0,T]\times \mathbb{R}^n.
\end{equation}

\medskip 
Since $|u_{\varepsilon}|^2\in L^{\infty}((0,\infty); L^2(\mathbb{R}))\hookrightarrow L^2([0,T]; L^2(\mathbb{R}))$ is bounded uniformly in $\varepsilon$, one gets
\begin{equation}
|u_{\varepsilon}|^2\overset{*}{\rightharpoonup} |\psi|^2 \quad  \text{in} \quad  L^2([0,T];L^2(\mathbb{R}))
\end{equation}
by reflexivity, so $\Gamma=|u|^2$.

\medskip 
The equality $\varLambda=\frac{\beta}{\lambda}u\Gamma$ follows by a standard argument. 
\end{proof}

\section{Proof of the results}

\subsection{\textbf{Proof of Theorem \ref{Th1}}}
\begin{proof}
As the transport solution in \eqref{DuhamelEq} is not easy to deal with, because of $\varepsilon\to0$ and the spacial derivative in the nonlinearity, then one can rewrite the transport equation in \eqref{S-Benney} as
\begin{equation}
\varepsilon\partial_t\Big(v-\frac{\beta}{\lambda}\vert u\vert^2\Big)+\lambda\partial_x\Big(v-\frac{\beta}{\lambda}\vert u\vert^2\Big)=-\frac{\varepsilon\beta}{\lambda}\partial_t\vert u\vert^2,
\end{equation}  
and with $w=v-\frac{\beta}{\lambda}\vert u\vert^2$ we have the solutions
\begin{equation}
w_{\varepsilon}(x,t)=\cT_{\varepsilon}(t)w_{0\varepsilon}-\tfrac{\beta}{\lambda}\int_0^{t}\cT_{\varepsilon}(t-s)\partial_t\vert u_{\varepsilon}(x,s)\vert^2ds.
\end{equation}

\medskip 
Notice that $$\cT_{\varepsilon}(t)\partial_t=-\frac{\varepsilon}{\lambda}\partial_t\cT_{\varepsilon}(t),$$
consequently
\begin{equation}\label{TransportSolution}
w_{\varepsilon}(x,t)=w_{0\varepsilon}\Big(x-\frac{\lambda}{\epsilon}t\Big)+\frac{\varepsilon\beta}{\lambda^2}\Big[1-\cT_{\varepsilon}(t)\Big] \vert u_{\varepsilon}(x,t)\vert^2
\end{equation}
where
\begin{equation}
 w_{0\varepsilon}=v_{0\varepsilon}-\frac{\beta}{\lambda}\vert u_{0\varepsilon}\vert^2.
\end{equation}

As $\Vert u_\varepsilon(\cdot,t)\Vert_{L_x^\infty}\lesssim\Vert u_\varepsilon(\cdot,t)\Vert_{H^1}$ because $H^1(\mathbb{R})\hookrightarrow C_{\infty}(\mathbb{R})$, then the first part of the theorem is now a immediate consequence of \eqref{TransportSolution}, the invariance by translation of $\Vert \cdot\Vert_{L^p(\mathbb{R})}$ and Lemma \ref{apriori-estimate-H1L-hypotheses}.

\medskip 
On the other hand, $$u_\varepsilon v_\varepsilon-\tfrac{\beta}{\lambda}u_\epsilon \vert u_\varepsilon  \vert^2=u_\varepsilon (v_\varepsilon-\tfrac{\beta}{\lambda} \vert u_\varepsilon  \vert^2),$$ then 
\begin{equation}
\big\| u_\varepsilon v_\varepsilon-\tfrac{\beta}{\lambda}u_\epsilon \vert u_\varepsilon  \vert^2\big\|_{L_T^{4/3}L_x^1}\lesssim T^{3/4}\big\| v_\varepsilon-\tfrac{\beta}{\lambda}\vert u_\epsilon \vert^2\big\|_{L_T^{\infty}L_x^2}
\end{equation}
follows by the Cauchy-Schwartz inequality and the uniform boundedness of $u_\varepsilon$ in the energy space again, due to the hypothesis.   

\medskip 
Now the proof is finished because of the smoothing effect:
$$\Big\|\int_0^tS(t-s)[u_\varepsilon v_\varepsilon-\tfrac{\beta}{\lambda}u_\varepsilon \vert u_\varepsilon\vert^2](x,s)ds \Big\|_{L_T^\infty L_x^2}\lesssim \big\| u_\varepsilon v_\varepsilon-\tfrac{\beta}{\lambda}u_\varepsilon \vert u_\varepsilon\vert^2\big\|_{L_T^{4/3}L_x^1} \lesssim  T^{3/4}O(\epsilon).$$
\end{proof}

\subsection{\textbf{Proof of Corollary \ref{Th1-corollary}}}

\begin{proof}
We have that 
\begin{equation}
\Big\|u_{\varepsilon}-S(t)u_{0\epsilon}+ i\frac{\alpha \beta}{\lambda}\int_0^tS(t-s)u_{\varepsilon}|u_{\varepsilon}|^2ds\Big\|_{L^{\infty}_TL^2_x}\lesssim T^{3/4}O(\varepsilon),
\end{equation}
so it is enough to verify that
\begin{equation}
\Big\|u-S(t)u_{0\epsilon}+ i\frac{\alpha \beta}{\lambda}\int_0^tS(t-s)u_{\varepsilon}|u_{\varepsilon}|^2ds\Big\|_{L^{\infty}_TL^2_x}\longrightarrow0
\end{equation}
as $\varepsilon\to0$, and apply the triangular inequality. Here
\begin{equation}
u(x,t)=S(t)u_0(x)- i\frac{\alpha \beta}{\lambda}\int_0^tS(t-s)u(x,s)|u(x,s)|^2\,ds,
\end{equation}
so it is only necessary to see for the nonlinear part. 

\medskip 
We have
\begin{equation}
\Big\vert \int_0^tS(t-s)u_{\varepsilon}|u_{\varepsilon}|^2ds\Big\vert^2\leq \Big(\int_0^t\vert u_{\varepsilon}\vert |u_{\varepsilon}|^2ds\Big)^2\leq \Big(\int_0^t\vert \psi\vert |\psi|^2ds\Big)^2
\end{equation}
and
\begin{align}
\Big(\int_{-\infty}^{+\infty}\Big(\int_0^t\vert \psi\vert |\psi|^2ds\Big)^2dx\Big)^{1/2}&\leq \int_0^t\Vert\vert \psi\vert |\psi|^2\Vert_{L^2_x}ds\\\notag
&\leq\int_0^t\Vert \psi\Vert^2_{L^\infty_x}\Vert \psi\Vert_{L^2_x}ds\\\notag
&\leq T\Vert \psi\Vert^2_{L^\infty_TL^\infty_x}\Vert \psi\Vert_{L^\infty_TL^2_x}.
\end{align}
As $u_\varepsilon\to u$ a.e, because of Theorem \ref{Th:WC}, the dominated convergence theorem give us
\begin{equation}
\lim_{\varepsilon\to0}\Big\Vert \int_0^tS(t-s)u_{\varepsilon}|u_{\varepsilon}|^2ds\Big\Vert_{L^\infty_TL^2_x}=\Big\Vert \int_0^tS(t-s)u|u|^2ds\Big\Vert_{L^\infty_TL^2_x}.
\end{equation}
Hence, we can conclude that  $\lim\limits_{\varepsilon\to0}\Vert u_{\varepsilon}-u\Vert_{L^\infty_TL^2_x}=0$ and the convergence $\lim\limits_{\varepsilon\to0}\Vert v_{\varepsilon}-\tfrac{\beta}{\lambda}\vert u\vert^2\Vert_{L^\infty_TL^2_x}=0$ follows by a similar argument. 
\end{proof}

\subsection{\textbf{Proof of Theorem \ref{Th2}}}
\begin{proof}
First we prove that the compatibility \eqref{Comp} is achieved.
$$\Vert v_{0\varepsilon}-\tfrac{\beta}{\lambda}\vert u_{0\varepsilon}\vert^2\Vert_{L^2}\leq \Vert v_{0\varepsilon}-\tfrac{\beta}{\lambda}\vert u_{0}\vert^2\Vert_{L^2}+\Vert \tfrac{\beta}{\lambda}\vert u_{0}\vert^2-\tfrac{\beta}{\lambda}\vert u_{0\varepsilon}\vert^2\Vert_{L^2}$$
and
\begin{align*}
\Vert \tfrac{\beta}{\lambda}\vert u_{0}\vert^2-\tfrac{\beta}{\lambda}\vert u_{0\varepsilon}\vert^2\Vert_{L^2}&=\vert\tfrac{\beta}{\lambda}\vert\Vert (\vert u_{0}\vert-\vert u_{0\varepsilon}\vert)(\vert u_{0}\vert+\vert u_{0\varepsilon}\vert)\Vert_{L^2}\\
&\lesssim\Vert \vert u_{0}\vert-\vert u_{0\varepsilon}\vert \Vert_{L^2}\\
&\leq\Vert u_{0}- u_{0\varepsilon} \Vert_{L^2},
\end{align*}
so
$$\lim_{\varepsilon\to0}\Vert v_{0\varepsilon}-\tfrac{\beta}{\lambda}\vert u_{0\varepsilon}\vert^2\Vert_{L^2}=0.$$

The same argument gives us that
\begin{align*}
\Vert v_{\varepsilon}-\tfrac{\beta}{\lambda}\vert u\vert^2\Vert_{L^2}&\leq \Vert v_{\varepsilon}-\tfrac{\beta}{\lambda}\vert u_{\varepsilon}\vert^2\Vert_{L^2}+\vert\tfrac{\beta}{\lambda}\vert\Vert \vert u_{\varepsilon}\vert^2-\vert u\vert^2\Vert_{L^2}\\
&\lesssim \Vert v_{\varepsilon}-\tfrac{\beta}{\lambda}\vert u_{\varepsilon}\vert^2\Vert_{L^2}+\Vert u_{\varepsilon}-u\Vert_{L^2},
\end{align*}
then 
\begin{equation}
\lim_{\varepsilon\to0}\Vert v_{\varepsilon}-\tfrac{\beta}{\lambda}\vert u\vert^2\Vert_{L^2}=0.
\end{equation}

Now let's check the $H^1$-convergence. First let's notice that as in the previous proof 
$$\Big\|\partial_x\int_0^tS(t-s)[u_\varepsilon v_\varepsilon-\tfrac{\beta}{\lambda}u_\varepsilon \vert u_\varepsilon\vert^2](x,s)ds \Big\|_{L_T^\infty L_x^2}\lesssim \big\| \partial_x [u_\varepsilon (v_\varepsilon-\tfrac{\beta}{\lambda} \vert u_\varepsilon  \vert^2)]\big\|_{L_T^{4/3}L_x^1}$$
and
\begin{align*}
 \big\| \partial_x [u_\varepsilon (v_\varepsilon-\tfrac{\beta}{\lambda} \vert u_\varepsilon  \vert^2)]\big\|_{L_T^{4/3}L_x^1}&\leq \big\| \partial_xu_\varepsilon (v_\varepsilon-\tfrac{\beta}{\lambda} \vert u_\varepsilon  \vert^2)\big\|_{L_T^{4/3}L_x^1}+\big\|  u_\varepsilon \partial_x(v_\varepsilon-\tfrac{\beta}{\lambda} \vert u_\varepsilon  \vert^2)\big\|_{L_T^{4/3}L_x^1}\\
 &\lesssim T^{3/4}\Big(\big\| v_\varepsilon-\tfrac{\beta}{\lambda}\vert u_\varepsilon \vert^2\big\|_{L_T^{\infty}L_x^2}+\big\| \partial_x(v_\varepsilon-\tfrac{\beta}{\lambda}\vert u_\varepsilon \vert^2)\big\|_{L_T^{\infty}L_x^2}\Big),
 \end{align*}
so we need to look for $\partial_xw_\varepsilon$ as in the proof of Theorem \ref{Th1}.
$$\partial_xw_\varepsilon(x,t)=\partial_xw_{0\varepsilon}\Big(x-\frac{\lambda}{\epsilon}t\Big)+\frac{\varepsilon\beta}{\lambda^2}\Big[1-\cT_{\varepsilon}(t)\Big]\partial_x \vert u_{\varepsilon}(x,t)\vert^2,$$
then
\begin{equation}
\Vert \partial_xw_\varepsilon(\cdot,t)\Vert_{L^2}=O(\varepsilon)
\end{equation}
and therefore 
\begin{equation}
\Big\|\partial_x\int_0^tS(t-s)[u_\varepsilon v_\varepsilon-\tfrac{\beta}{\lambda}u_\varepsilon \vert u_\varepsilon\vert^2](x,s)ds \Big\|_{L_T^\infty L_x^2} =  T^{3/4}O(\varepsilon).
\end{equation}

Now, this means that
\begin{equation}
\Big\|\partial_x\Big(u_{\varepsilon}-S(t)u_{0\epsilon}+ i\frac{\alpha \beta}{\lambda}\int_0^tS(t-s)u_{\varepsilon}|u_{\varepsilon}|^2ds\Big)\Big\|_{L^{\infty}_TL^2_x}\lesssim T^{3/4}O(\varepsilon)
\end{equation}
and the proof is finished. 
\end{proof}

\subsection{\textbf{Proof of Theorem \ref{Head-Th}}}
\begin{proof}
This theorem is a clear consequence of the procedure used to obtain the before results in the context of  the Benney system \eqref{S-Benney}.
\end{proof}

\section{Final remarks} \label{Final-Remarks}

\begin{remark}[Local well-posedness]\label{subsection-lwp-GBenney}
Notice that if it is considered $|\lambda|=|\lambda'|=1$  in  system \eqref{S-GBenney}, the same theory developed in \cite{G-T-V} allows to conclude 
local well-posedness for this system for initial data $(u_0, v_0, z_0) \in H^s\times H^{\kappa} \times  H^{\kappa}$ with $(s, \kappa)$ satisfying the conditions in \eqref{wp-condition-Z-1d}. This fact follows immediately observing that the linear parts corresponding to the transport equations are not coupled.
\end{remark}

\begin{remark}[Conservation laws and global solutions]
Formally, system  \eqref{S-GBenney} satisfies the following conservation laws:
\begin{equation}\label{Mass-GBenney}
\widetilde{\cM}(t):=\int_{-\infty}^{+\infty}|u|^{2}dx=\widetilde{\cM}(0)\quad \big(\text{mass}\big),
\end{equation}

\begin{equation}\label{Momment-GBenney}
\widetilde{\cK}(t):=\int_{-\infty}^{+\infty}\Big[ \tfrac{\alpha}{2\beta}v^2 + \tfrac{\alpha'}{2\beta'}z^2 +\tfrac{1}{\varepsilon} \text{Im}\, u\bar{u}_x\Big]dx=\widetilde{\cK}(0)
\quad \big(\text{moment}\big)
\end{equation}
and
\begin{equation}\label{Energy-GBenney}
\widetilde{\cE}(t):=\int_{-\infty}^{+\infty}\Big(|u_x|^2 + \tfrac{\tau}{2}|u|^4 + \alpha v|u|^2 + \alpha' z|u|^2 
-\tfrac{\alpha \lambda}{2\beta}v^2-\tfrac{\alpha' \lambda'}{2\beta'}z^2\Big)dx=\widetilde{\cE}(0)
\quad \big(\text{energy}\big),
\end{equation}
with $\beta, \beta' \neq 0$ and for all $0\le t< T^*_{\varepsilon}$, where $T^*_{\varepsilon}$ is the maximal time of existence for the respective solution. 

\medskip
Since the region of regularity in described in Remark \ref{subsection-lwp-GBenney} includes the case $H^1\times L^2 \times L^2$  when $|\lambda|=|\lambda'|=1$, the conservation laws above ensure global well-posedness in this case. 
\end{remark}

\begin{remark}
We observe that the same process used to prove Lemma \ref{apriori-estimate-H1L-hypotheses} ensure that the condition \eqref{bounded-u-epsilon} in Theorem \ref{Head-Th} is valid in the case:
$$\tfrac{\alpha \lambda }{\beta}<0\quad \text{and}\quad \tfrac{\alpha' \lambda' }{\beta'}<0.$$
\end{remark}

\medskip 
We are currently adapting the ideas used in this work to obtain more accurate results in the same direction for the Zakharov system \eqref{Z} in higher dimensions.

\end{document}